\documentclass[11pt,a4paper]{article}
\usepackage{a4wide}
\usepackage[utf8]{inputenc}
\usepackage[T1]{fontenc}
\usepackage{amsmath,amssymb,amsfonts}
\usepackage{mathtools}
\mathtoolsset{showonlyrefs=true}
\usepackage{amsthm}
\usepackage{graphicx}
\usepackage{booktabs}
\usepackage{algorithmic}
\usepackage{xcolor}
\usepackage{cite}
\usepackage{url}
\usepackage[hidelinks]{hyperref}
\usepackage{microtype}
\usepackage{mathrsfs}
\DeclareMathOperator{\diag}{diag}

\DeclareMathOperator{\grad}{grad}

\DeclareMathOperator{\qf}{qf}

\DeclareMathOperator{\re}{Re}

\let\skew\relax
\DeclareMathOperator{\skew}{skew}
\DeclareMathOperator{\Skew}{Skew}

\DeclareMathOperator{\St}{St}
\DeclareMathOperator{\her}{her}
\DeclareMathOperator{\Her}{Her}

\DeclareMathOperator{\tr}{tr}
\def\D{\mathrm{D}}
\def\id{\mathrm{id}}

\def\C{\mathbb{C}}
\def\F{\mathbb{F}}
\def\H{\mathbb{H}}
\def\R{\mathbb{R}}

\newtheorem{proposition}{Proposition}

\newtheorem{lemma}{Lemma}
\newtheorem{corollary}{Corollary}
\newtheorem{remark}{Remark}
\date{}
\title{\LARGE \bf
Riemannian optimization framework\\
on the generalized quaternionic Stiefel manifold
}
\author{Hiroyuki Sato\thanks{This work was supported by JSPS KAKENHI Grant Numbers JP25K07125, JP25K03082, and
JP24K14985.}
\thanks{Hiroyuki Sato is with Department of Mathematical Sciences,
        Ritsumeikan University, Shiga, Japan
        {\tt\small hsato@fc.ritsumei.ac.jp}}
}
    \makeatletter
    \let\NAT@parse\undefined
    \makeatother
\begin{document}
\maketitle
\begin{abstract}
This paper introduces the generalized quaternionic Stiefel manifold and studies its geometry for Riemannian optimization. We clarify its relationships with existing manifolds, especially the real generalized Stiefel manifold and the quaternionic Stiefel manifold, and derive explicit expressions for several geometric quantities on the proposed manifold. In particular, the generalized quaternionic Stiefel manifold is regarded as a real Riemannian manifold, and expressions for the tangent space, normal space, the orthogonal projection onto the tangent space, a retraction, and a vector transport on this manifold are derived. Some numerical experiments for the generalized quaternionic eigenvalue problem and a version of quaternionic canonical correlation analysis are performed by Riemannian optimization methods to demonstrate the viability of the proposed optimization framework.
\end{abstract}
\section{Introduction}
Riemannian optimization, which is optimization on Riemannian manifolds, is known as an effective tool for solving a wide class of optimization problems~\cite{Absil2008,Boumal2023,Sato2021}.
It can serve as an alternative to Euclidean optimization, i.e., the usual optimization in Euclidean spaces, and can also solve some problems that are hard to solve within the framework of Euclidean optimization; e.g., optimization on the Grassmann manifold can deal with linear subspaces of the Euclidean space as decision variables~\cite{Absil2008,Edelman1998}.
Let $n$ and $p$ be positive integers satisfying $p \leq n$.
The (real) Stiefel manifold, which is defined as\footnote{See Section~\ref{sec:notation} for detailed notation in this paper.}
\begin{equation}
\St(p,\R^n) \coloneqq \{X \in \R^{n \times p} \mid X^{\top}X = I_p\},
\end{equation}
is one of the most important Riemannian manifolds in the field of Riemannian optimization.
Optimization on the Stiefel manifold and on more generalized manifolds has been extensively studied~\cite{Edelman1998}.
There are many applications of optimization on the Stiefel manifold.
One example is the principal component analysis (PCA).
It is formulated as an optimization problem on the product of two Stiefel manifolds~\cite{Absil2008,SatoIwai2013}.
Furthermore, the Stiefel manifold has been generalized in the following two ways.
1) The Stiefel manifold $\St(p,\R^n)$ can be regarded as the set of orthonormal $p$-frames in $\R^n$ by identifying each $p$-frame $(x_1, x_2, \dots, x_p)$ with the $n \times p$ matrix $X = \begin{bmatrix}
    x_1 & x_2 & \cdots & x_p
\end{bmatrix}$, where orthonormality is defined with respect to the standard inner product in $\R^n$. Therefore, a direction to generalize the Stiefel manifold is to endow $\R^n$ with a more general inner product defined via a symmetric positive definite matrix $G \in \R^{n \times n}$.
The resulting manifold, called the generalized (real) Stiefel manifold~\cite{SatoAihara2019,ShustinAvron2023}, is
\begin{equation}
\St_G(p,\R^n) \coloneqq \{X \in \R^{n \times p} \mid X^{\top}GX = I_p\}.
\end{equation}
The Stiefel manifold $\St(p,\R^n)$ encodes the orthonormality of data with respect to the standard inner product, whereas the generalized Stiefel manifold $\St_G(p,\R^n)$ is useful for orthonormality with respect to the inner product defined by $G$.
One application is the generalized eigenvalue problem (GEVP)~\cite{ShustinAvron2023}.
Another example is the canonical correlation analysis (CCA)~\cite{Yger2012}, which is formulated as an optimization problem on the product of two generalized Stiefel manifolds of different dimensions.
2) We can consider $p$-frames in $\F^n$, where $\F \in \{\R, \C, \H\}$.
Here, $\C$ and $\H$ are the sets of all complex numbers and quaternions, respectively.
When $\F = \C$, the resulting manifold
\begin{equation}
\label{eq:complexStiefel}
\St(p,\C^n) \coloneqq \{X \in \C^{n \times p} \mid X^{H}X = I_p\}
\end{equation}
is called the complex Stiefel manifold.
When $\F = \H$, which is the main focus of this paper, the resulting manifold
\begin{equation}
\label{eq:quaternionStiefel}
\St(p,\H^n) \coloneqq \{X \in \H^{n \times p} \mid X^H X = I_p\}
\end{equation}
is called the quaternionic Stiefel manifold (also called the quaternion Stiefel manifold).
Here, $\cdot^H$ in~\eqref{eq:complexStiefel} (resp.~\eqref{eq:quaternionStiefel}) denotes the Hermitian conjugate of a complex (resp. quaternionic) matrix (see Section~\ref{sec:notation} for details).
Optimization on $\St(p,\H^n)$ has been studied very recently~\cite{HuangJiaLi2025,WangYang2026}.
An example of optimization on the complex (resp. quaternionic) Stiefel manifold is the PCA of the data expressed as complex numbers (resp. quaternions).
In this paper, combining the above two perspectives, we introduce the concept of the generalized quaternionic Stiefel manifold and study optimization on it.\footnote{Studies on the generalized complex Stiefel manifold are found in~\cite{sedano2022isometry}.}
Furthermore, from the above discussion, we can expect that optimization on this manifold has applications such as the GEVP and CCA for the given quaternionic data (matrices).
Here, we emphasize that a quaternion can well describe $3$- or $4$-dimensional data.
The organization of the paper is as follows.
In Section~\ref{sec:notation}, we describe the notation used in this paper and review some basic properties of quaternions.
In Section~\ref{sec:gqStiefel}, we investigate the geometry of the generalized quaternionic Stiefel manifold as a real Riemannian manifold.
Section~\ref{sec:applications} provides some numerical results of optimization on the proposed manifold, taking the GEVP and quaternionic CCA as applications.
Section~\ref{sec:conclusion} concludes the paper.
\section{Notation and review of quaternions}
\label{sec:notation}
Let $\H$ be the set of all quaternions.
A quaternion $q \in \H$ is a number expressed as
$q = q_0 + q_1 i + q_2 j + q_3 k$
with $q_0, q_1, q_2, q_3 \in \R$, where $i, j, k$ are the imaginary units satisfying
$i^2 = j^2 = k^2 = ijk = -1$.
Two quaternions do not generally commute, e.g., $ij = k \neq -k = ji$.
The conjugate, absolute value, and real part of $q = q_0 + q_1 i + q_2 j + q_3 k$ are defined by
$\bar{q} \coloneqq q_0 - q_1 i - q_2 j - q_3 k$, $\re(q) \coloneqq q_0$,
and $|q| \coloneqq \sqrt{q\bar{q}} = \sqrt{\bar{q}q} = \sqrt{q_0^2 + q_1^2 + q_2^2 + q_3^2}$, respectively.
Note that $\H$ is not a (commutative) field.
We regard $\H^n$ as a right $\H$-vector space, where the scalar product of a vector $x \in \H^n$ and a scalar $q \in \H$ is defined by $xq$.
Then, the standard inner product in $\H^n$ is defined by
$\langle x, y\rangle \coloneqq x^H y$ for $x, y \in \H^n$,
where $x^H$ is the Hermitian conjugate of $x$ (see below).
This satisfies, for any vectors $x, y, z \in \H^n$ and scalar $q \in \H$,
1) $\langle y, x\rangle = \overline{\langle x, y\rangle}$ (conjugate symmetry);
2) $\langle x, y + z\rangle = \langle x, y\rangle + \langle x, z\rangle$ and $\langle x, yq\rangle = \langle x, y\rangle q$ (right linearity in the second argument);
and 3) $\langle x, x\rangle \geq 0$, and $\langle x, x\rangle = 0 \iff x = 0$ (positive definiteness).
The induced norm of $x$ is then $\|x\| \coloneqq \sqrt{\langle x, x\rangle} = \sqrt{x^H x}$.
Throughout the paper, let $m, n, p$ be positive integers and assume $p \leq n$.
Let $\F$ be one of $\R$, $\C$, or $\H$.
The transpose of a matrix $A \in \F^{m \times n}$ is denoted as $A^{\top}$.
The Hermitian conjugate of $A \in \F^{m \times n}$ is
$A^H \coloneqq \bar{A}^{\top}$,
where $\bar{A}$ is the conjugate of $A$, whose $(r, s)$ element is the conjugate of the $(r, s)$ element of $A$.
Specifically,
the Hermitian conjugate of a quaternionic matrix $A = A_0 + A_1 i + A_2 j + A_3 k \in \H^{m \times n}$ with $A_0, A_1, A_2, A_3 \in \R^{m \times n}$ is $A^H = A_0^{\top} - A_1^{\top} i - A_2^{\top} j - A_3^{\top} k$.
For $A \in \H^{l \times m}$ and $B \in \H^{m \times n}$, we have $(AB)^H = B^H A^H$.
However, we have to note that $\overline{AB}$ is not necessarily equal to $\bar{A}\bar{B}$, and that $(AB)^{\top}$ is not necessarily equal to $B^{\top}A^{\top}$ due to the noncommutativity of quaternions.
The set $\H^{m \times n}$ can be endowed with a vector space structure over $\R$, and its (real) dimension is $4mn$ since the map $\psi \colon \H^{m \times n} \to (\R^{m \times n})^4$ defined by
$\psi(A) \coloneqq (A_0, A_1, A_2, A_3)$
for $A = A_0 + A_1 i + A_2 j + A_3 k$ is a linear isomorphism.
From these observations, we are led to the
identifications
$\H^{m \times n}
\simeq (\R^{m \times n})^4 \simeq \R^{4mn}$.
The standard inner product in the $4mn$-dimensional real vector space $\H^{m \times n}$ can be defined via the one in $(\R^{m \times n})^4$ as
$\langle X, Y\rangle \coloneqq \re(\tr(X^H Y)) = \sum_{l = 0}^3 \tr(X_l^{\top}Y_l)$
for $X = X_0 + X_1 i + X_2 j + X_3 k, Y = Y_0 + Y_1 i + Y_2 j + Y_3 k \in \H^{m \times n}$.
For $A \in \H^{m \times n}$ and $B \in \H^{n \times m}$,
we note a general formula
\begin{equation}
\re(\tr(AB)) = \re(\tr(BA))
\end{equation}
holds, while $\tr(AB)$ and $\tr(BA)$ are not necessarily equal due to the noncommutativity of quaternions.
A quaternionic matrix $A \in \H^{n \times n}$ is said to be Hermitian if $A^H = A$ and skew-Hermitian if $A^H = -A$.
The sets of all $n \times n$ Hermitian and skew-Hermitian quaternionic matrices are denoted by $\Her(n)$ and $\Skew(n)$, respectively.
The Hermitian and skew-Hermitian parts of $A \in \H^{n \times n}$ are denoted by
$\her(A) \coloneqq (A + A^H)/2$ and $\skew(A) \coloneqq (A - A^H)/2$,
respectively.
An important property is that any $S \in \Her(n)$ and $T \in \Skew(n)$ satisfy
$\re(\tr(ST)) = 0$
since we have
$\re(\tr(ST)) = \re(\tr((ST)^H))
= -\re(\tr(TS))
= -\re(\tr(ST))$.
\section{Generalized quaternionic Stiefel manifold}
\label{sec:gqStiefel}
In what follows, let $G \in \Her(n)$ be Hermitian positive definite, i.e., $G^H = G$ and $x^H Gx > 0$ for any nonzero $x \in \H^n$.
We define a quaternionic inner product as
\begin{equation}
\langle x, y\rangle_G \coloneqq x^H Gy, \quad x, y \in \H^n.
\end{equation}
The inner product $\langle \cdot, \cdot\rangle_G$ thus defined is shown to satisfy the three properties of a quaternionic inner product stated in the previous section.
Two vectors $x, y \in \H^n$ are said to be orthogonal if $0 = \langle x, y\rangle_G = x^H Gy$, and $x$ is said to be a unit vector if $1 = \|x\|_G \coloneqq \sqrt{\langle x, x\rangle_G}$.
Given this inner product and the induced norm $\|\cdot\|_G$, we consider an orthonormal $p$-frame $(x_1, x_2, \dots, x_p)$, i.e., $\langle x_r, x_s\rangle_G = \delta_{rs}$, where $\delta_{rs}$ is Kronecker's delta.
Expressing the $p$-frame by a matrix $X \coloneqq \begin{bmatrix}
x_1 & x_2 & \cdots & x_p
\end{bmatrix} \in \H^{n \times p}$, we have $X^H G X = I_p$.
Motivated by this observation, we define the \textit{generalized quaternionic Stiefel manifold} with respect to $G$ as
\begin{equation}
\St_G(p,\H^n) \coloneqq \{X \in \H^{n \times p} \mid X^H G X = I_p\}.
\end{equation}
In the remainder of this section, we show that $\St_G(p,\H^n)$ is a (real) embedded submanifold of the (real) $4np$-dimensional manifold $\H^{n \times p}$ and derive specific formulas for several geometric concepts, including the tangent and normal spaces.
For positive definite $G \in \Her(n)$, there exists a unique positive definite Hermitian square root $\sqrt{G}$ of $G$ such that $\sqrt{G}^2 = G$~\cite{WangYang2026,Zhang1997}.
Then, for $X \in \H^{n \times p}$, we have the equivalence
$X \in \St_G(p,\H^n)  \iff
X^H GX = I_p \iff (\sqrt{G}X)^H(\sqrt{G}X) = I_p \iff \sqrt{G}X \in \St(p,\H^n)$,
where $X^H \sqrt{G} =X^H\sqrt{G}^H = (\sqrt{G}X
)^H$ is used.
Therefore, defining $\Phi \colon \H^{n \times p} \to \H^{n \times p}$ as $\Phi(Y) \coloneqq \sqrt{G}^{-1}Y$, we have
$\St_G(p,\H^n) = \Phi(\St(p,\H^n))$.
We can regard the map $\Phi \colon \H^{n \times p} \to \H^{n \times p}$ as a map $\R^{4np} \to \R^{4np}$ by identifying $\Phi$ with $\psi \circ \Phi \circ \psi^{-1}$.
Since $\Phi$ is a linear isomorphism on $\R^{4np}$, it is a diffeomorphism.
\begin{proposition}
The generalized quaternionic Stiefel manifold $\St_G(p,\H^n)$ is a $p(4n - 2p + 1)$-dimensional real submanifold of $\H^{n \times p}$.
Strictly speaking, $\psi(\St_G(p,\H^n))$ is an embedded submanifold of $\psi(\H^{n \times p}) \simeq \R^{4np}$.
\end{proposition}
\begin{proof}
According to~\cite{WangYang2026}, $\St(p,\H^n)$ is an embedded submanifold of $\H^{n \times p} \simeq \R^{4np}$, i.e., $\psi(\St(p,\H^n))$ is a submanifold of $\psi(\H^{n \times p}) \simeq \R^{4np}$.
From general theory~\cite{Tu2011}, $(\psi \circ \Phi \circ \psi^{-1})(\psi(\St(p,\H^n))) = \psi(\St_G(p,\H^n))$ is a submanifold of $\R^{4np}$ since $\psi \circ \Phi \circ \psi^{-1}$ is a diffeomorphism.
Further, the dimension of $\St_G(p,\H^n)$ is the same as that of $\St(p,\H^n)$, which is $p(4n-2p+1)$~\cite[Theorem 3.1]{WangYang2026}.
\end{proof}
\begin{proposition}
The tangent space of $\St_G(p,\H^n)$ at $X$ is
$T_X\!\St_G(p,\H^n) = \{\xi \in \H^{n \times p} \mid \xi^H G X + X^H G \xi = 0\}$.
Furthermore, letting $X_{\perp} \in \H^{n \times (n-p)}$ be an arbitrary matrix satisfying $X_{\perp}^H GX_{\perp} = I_{n-p}$ and $X^H GX_{\perp} = 0$, we have
$T_X\!\St_G(p,\H^n\hspace{-.1em})
\!=\! \{XB \hspace{-.1em}+\hspace{-.1em} X_{\perp}C \hspace{.1em}|\hspace{.1em} B \!\in\! \Skew(p), C \!\in\! \H^{(n-p) \hspace{.1em}\!\times\!\hspace{.1em} p\hspace{-.1em}}\}$.
\end{proposition}
\begin{proof}
For any $Y, Z \in \H^{n \times p}$, it follows immediately from $\Phi(Y) = \sqrt{G}^{-1}Y$ that $\D \Phi(Y)[Z] = \sqrt{G}^{-1}Z$.
Here, for $X \in \St_G(p,\H^n)$, let $Y \coloneqq \Phi^{-1}(X) = \sqrt{G}X \in \St(p,\H^n)$.
Then, we have
\begin{align}
T_{X}\!\St_G(p,\H^n)={}&
T_{\Phi(Y)}\Phi(\St(p,\H^n))\\
={}&
 \D \Phi(Y)(T_Y\! \St(p,\H^n))
=
 \{\sqrt{G}^{-1}\eta \mid \eta \in T_Y\!\St(p,\H^n)\}\notag\\
={}&
 \{\xi \in \H^{n \times p} \mid \sqrt{G}\xi \in T_{\sqrt{G}X}\!\St(p,\H^n)\}\\
={}&
 \{\xi \in \H^{n \times p} \mid (\sqrt{G}\xi)^H(\sqrt{G}X) + (\sqrt{G}X)^H(\sqrt{G}\xi) = 0\}\notag\\
={}& \{\xi \in \H^{n \times p} \mid \xi^HGX + X^HG\xi = 0\}.
\end{align}
Analogous to the relationship between $X$ and $Y$, letting $Y_{\perp} \coloneqq \sqrt{G}X_{\perp}$, we have $Y_{\perp}^H Y_{\perp} = X_{\perp}^H GX_{\perp} = I_{n-p}$ and $Y^HY_{\perp} = X^HGX_{\perp} = 0$.
Thus, from~\cite{WangYang2026}, we have $T_Y\!\St(p,\H^n) \!=\! \{YB + Y_{\perp}C \!\mid\! B \in \Skew(p), C \in \H^{(n-p) \times p}\}$.
It follows that
\begin{align}
T_X\!\St_G(p,\H^n)
={}& \D \Phi(Y)(T_Y\!\St(p,\H^n))\\
={}& \{\sqrt{G}^{-1}(YB + Y_{\perp}C) \mid B \in \Skew(p), C \in \H^{(n-p) \times p}\}\\
={}& \{XB + X_{\perp}C\mid B \in \Skew(p), C \in \H^{(n-p) \times p}\}.
\end{align}
This completes the proof.
\end{proof}
We consider a Riemannian metric in the quaternionic matrix space $\H^{n \times p}$ defined by
\begin{equation}
\label{eq:H_RiemannianMetric}
\langle \xi, \eta\rangle_X \coloneqq \re(\tr(\xi^H M_X \eta)), \quad \xi, \eta \in \H^{n \times p}, \ X \in \H^{n \times p},
\end{equation}
where the map $\H^{n \times p} \ni X \mapsto M_X \in \Her(n)$ is smooth and $M_X$ is positive definite for all $X$.
Then, we can endow $\St_G(p,\H^n)$ with the induced metric
\begin{equation}
\label{eq:RiemannianMetric}
\langle \xi, \eta\rangle_X \coloneqq \re(\tr(\xi^H M_X \eta))
\end{equation}
for $ X \in \St_G(p,\H^n)$ and $\xi, \eta \in T_X \! \St_G(p,\H^n)$.
With this Riemannian metric, $\St_G(p,\H^n)$ is a Riemannian submanifold of $\H^{n \times p}$.
\begin{remark}
The matrices $G$ and \(M_X\) have different roles.
The matrix $M_X$ defines the Riemannian metric
and is an algorithmic choice.
The choice $M_X\equiv G$ is a natural one and leads to the simple projection formula in Corollary 1 below, whereas other choices can also be used to incorporate preconditioning or canonical-type geometry.
Regarding the real case, such non-standard metrics are used as Riemannian preconditioners~\cite{ShustinAvron2023}; hence allowing a general $M_X$ in our setting is also expected to be useful from an algorithmic point of view.
We observe this in Section~\ref{subsec:GEVP}.
\end{remark}
\begin{proposition}
\label{prop:normal}
The normal space of $\St_G(p,\H^n)$ at $X$ is
$N_X\!\St_G(p,\H^n) = \{M_X^{-1}GXS \mid S \in \Her(p)\}$,
where the normal space $N_X\!\St_G(p,\H^n)$ is defined to be the orthogonal complement  of the tangent space $T_X\!\St_G(p,\H^n)$ with respect to the inner product~\eqref{eq:H_RiemannianMetric} in $T_X \H^{n \times p} = \H^{n \times p}$.
\end{proposition}
\begin{proof}
Let $X_{\perp} \in \H^{n \times (n-p)}$ satisfy $X_{\perp}^H GX_{\perp} = I_{n-p}$ and $X^H G X_{\perp} = 0$.
Then, any $Y \in \H^{n \times p}$ is written as
$Y = M_X^{-1}GX D + M_X^{-1}GX_{\perp}E$
for some $D \in \H^{p \times p}$ and $E \in \H^{(n-p) \times p}$.
For $\xi = X B + X_{\perp}C \in T_X\!\St_G(p,\H^n)$, we have
\begin{equation}
\langle \xi, Y\rangle_X = \re(\tr(\xi^H M_XY))
= \re(\tr(B^H D + C^H E)),
\end{equation}
where $B \in \Skew(p)$ and $C \in \H^{(n-p) \times p}$.
The condition $Y \in N_X\!\St_G(p,\H^n)$ is equivalent to $\langle \xi, Y\rangle_X = 0$ for any $\xi \in T_X\!\St_G(p,\H^n)$, i.e., $\re(\tr(B^H D + C^H E)) = 0$ for any $B \in \Skew(p)$ and $C \in \H^{(n-p) \times p}$.
We directly obtain $E = 0$.
Furthermore, since $\re(\tr(B^HD)) = \re(\tr(B^H \her(D))) + \re(\tr(B^H\skew(D))) = \re(\tr(B^H\skew(D)))$, this is equal to $0$ for all $B \in \Skew(p)$ if and only if $\skew(D) = 0$, i.e., $D \in \Her(p)$.
This completes the proof.
\end{proof}
Before proceeding with the orthogonal projection to the tangent space, we prepare the following lemma.
\begin{lemma}
\label{lem:Sylvester}
Let $K, L \in \Her(p)$ and assume that $K$ is positive definite.
Then, the Sylvester equation
\begin{equation}
\label{eq:Sylvester_general}
KS + SK = L
\end{equation}
for $S \in \H^{p \times p}$ has a unique solution.
Furthermore, the solution is Hermitian.
\end{lemma}
\begin{proof}
Since $K$ is Hermitian positive definite, it can be decomposed as
$K = U\Lambda U^H$,
where $U \in \H^{p \times p}$ is unitary ($U^H U = UU^H =  I_p$) and $\Lambda$ is a $p \times p$ diagonal matrix with positive diagonal elements $\lambda_1, \lambda_2, \dots, \lambda_p > 0$~\cite{WangYang2026,Zhang1997}.
Then, the Sylvester equation~\eqref{eq:Sylvester_general} reduces to $U\Lambda U^H S + SU\Lambda U^H = L$,
which is equivalent to
\begin{equation}
\label{eq:aux_Sylvester}
\Lambda (U^H SU) + (U^H SU)\Lambda = U^H L U.
\end{equation}
Let $T = [t_{rs}] \coloneqq U^H SU$.
Taking the $(r,s)$ elements of both sides of~\eqref{eq:aux_Sylvester} gives $\lambda_r t_{rs} + \lambda_s t_{rs} = (U^HLU)_{rs}$, which yields
\begin{equation}
\label{eq:tij}
t_{rs} =
(\lambda_r + \lambda_s)^{-1}(U^HLU)_{rs}.
\end{equation}
Therefore, the unique solution to \eqref{eq:Sylvester_general} is
$S = UTU^H$,
where the $(r,s)$ element of $T$ is given by~\eqref{eq:tij}.
Since $L$ and thus $U^H LU$ are Hermitian, so are $T$ and $S$.
\end{proof}
\begin{proposition}
Let $X \in \St_G(p,\H^n)$.
The orthogonal projection $P_X^{M_X} \colon \H^{n \times p} \to T_X\!\St_G(p,\H^n)$ onto the tangent space $T_X\!\St_G(p,\H^n)$ at $X$ is expressed as
\begin{equation}
\label{eq:proj_general}
P_X^{M_X}(Y) = Y - M_X^{-1}GXS,
\end{equation}
where $S \in \Her(p)$ is the solution to the Sylvester equation
\begin{equation}
\label{eq:Sylvester}
\her(X^H GM_X^{-1}GXS) = \her(X^H GY).
\end{equation}
\end{proposition}
\begin{proof}
Note that $K \coloneqq X^H G M_X^{-1}GX$ is Hermitian positive definite since $M_X^{-1}$ is positive definite and thus $x^HKx = (GXx)^HM_X^{-1}(GXx) > 0$ holds for any nonzero $x \in \H^p$.
Therefore, from Lemma~\ref{lem:Sylvester}, the Sylvester equation~\eqref{eq:Sylvester} has a unique solution.
Let $Z \coloneqq Y - M_X^{-1}GXS$, where $S \in \Her(p)$ satisfies~\eqref{eq:Sylvester}.
Then, we have
\begin{equation}
\her(X^H G Z) = \her(X^H GY) - \her(X^H G M_X^{-1}GXS) = 0,\notag
\end{equation}
which implies $Z \in T_X \!\St_G(p,\H^n)$.
Furthermore, $Y - Z = M_X^{-1}GXS \in N_X\!\St_G(p,\H^n)$.
Therefore, $Z$ is the orthogonal projection of $Y$ to $T_X\!\St_G(p,\H^n)$, completing the proof.
\end{proof}
\begin{corollary}
When $M_X \equiv G$ for all $X \in \St_G(p,\H^n)$, it holds that
\begin{equation}
\label{eq:proj_G}
P_X^G(Y) = Y - X\her(X^HGY).
\end{equation}
\end{corollary}
\begin{proof}
When $M_X \equiv G$, the left-hand side of the Sylvester equation~\eqref{eq:Sylvester} is computed as
\begin{equation}
\her(X^H GM_X^{-1}GXS) = \her(X^HGXS) = \her(S) = S.
\end{equation}
Thus, the solution to~\eqref{eq:Sylvester} follows as $S = \her(X^H GY)$.
Substituting this and $M_X = G$ into~\eqref{eq:proj_general} yields~\eqref{eq:proj_G}.
\end{proof}
The Euclidean gradient of a real-valued function $\bar{f} \colon \H^{n \times p} \to \R$ is defined to satisfy, for any $Z \in \H^{n \times p}$,
$\re(\tr(\nabla \bar{f}(X)^H Z)) = \lim_{t \to 0} (\bar{f}(X+tZ) - \bar{f}(X)) / {t}$,
if the limit exists.
Using the expression $X = X_0 + X_1 i + X_2 j + X_3 k$ with $X_0, X_1, X_2, X_3 \in \R^{n \times p}$, $\bar{f}$ can also be regarded as a function of $(X_0, X_1, X_2, X_3) \in (\R^{n \times p})^4$.
For $X_l$, $l = 0, 1, 2, 3$, let $\nabla_{X_l}\bar{f}(X)$ denote the Euclidean gradient of $X_l \mapsto \bar{f}(X)$ with $X_u$ for $u \in \{0, 1, 2, 3\} - \{l\}$ being fixed.
Then, we have $\nabla \bar{f}(X) = \nabla_{X_0}\bar{f}(X) + \nabla_{X_1}\bar{f}(X)i + \nabla_{X_2}\bar{f}(X)j + \nabla_{X_3}\bar{f}(X)k$.
Subsequently, we examine the Riemannian gradient of a function $f \colon \St_G(p,\H^n) \to \R$ with respect to the Riemannian metric~\eqref{eq:RiemannianMetric} induced from~\eqref{eq:H_RiemannianMetric}.
To this end, we first investigate the Riemannian gradient in $\H^{n \times p}$.
\begin{proposition}
Consider the Riemannian manifold $\H^{n \times p}$ endowed with the Riemannian metric~\eqref{eq:H_RiemannianMetric}.
Then, the Riemannian gradient of a smooth function $\bar{f} \colon \H^{n \times p} \to \R$ is given by
$\grad \bar{f}(X) = M_X^{-1}\nabla \bar{f}(X)$,
where $\nabla \bar{f}$ is the Euclidean gradient of $\bar{f}$.
\end{proposition}
\begin{proof}
For any $Z \in \H^{n \times p}$, we have
\begin{align}
\re(\tr(\nabla \bar{f}(X)^H Z)) &= \D \bar{f}(X)[Z]\\
&= \langle \grad \bar{f}(X), Z\rangle_X\\
&= \re(\tr(\grad \bar{f}(X)^HM_XZ))\\
&= \re(\tr((M_X \grad \bar{f}(X))^HZ)).
\end{align}
This implies $M_X \grad \bar{f}(X) = \nabla \bar{f}(X)$, i.e., $\grad \bar{f}(X) = M_X^{-1}\nabla \bar{f}(X)$.
\end{proof}
\begin{proposition}
\label{prop:grad}
The Riemannian gradient of smooth $f \colon \St_G(p,\H^n) \to \R$ with respect to~\eqref{eq:RiemannianMetric} is given by
$\grad f(X) = P^{M_X}_X(\grad \bar{f}(X)) = P_X^{M_X}(M_X^{-1}\nabla \bar{f}(X))$,
where $\bar{f} \colon \H^{n \times p} \to \R$ is a smooth extension of $f$, $P^{M_X}_X$ is the orthogonal projection~\eqref{eq:proj_general}, and $\nabla \bar{f}$ is the Euclidean gradient.
\end{proposition}
\begin{proof}
The statement of the proposition is a direct consequence of general theory (e.g., \cite{Absil2008}) for the Riemannian gradient on a Riemannian submanifold.
\end{proof}
In Riemannian optimization, a retraction is utilized to generate a sequence on the manifold in question~\cite{Absil2008}.
A retraction $R \colon T\!\St_G(p,\H^n) \to \St_G(p,\H^n)$ on $\St_G(p,\H^n)$ should be defined to satisfy $R_X(0) = X$ and $\D R_X(0) = \id_{T_X \!\St_G(p,\H^n)}$, where $R_X \colon T_X\!\St_G(p,\H^n) \to \St_G(p,\H^n)$ is the restriction of $R$ to the tangent space $T_X\!\St_G(p,\H^n)$ and $\id_{T_X\!\St_G(p,\H^n)}$ is the identity map in $T_X\!\St_G(p,\H^n)$.
\begin{proposition}
We define a map $R^G \colon T \!\St_G(p,\H^n) \to \St_G(p,\H^n)$ by
$R^G_X(\eta) \coloneqq \sqrt{G}^{-1}\qf(\sqrt{G}(X + \eta))$
for $X \in \St_G(p,\H^n)$ and $\eta \in T_X \!\St_G(p,\H^n)$, where $\qf(A)$ denotes the Q-factor of the QR decomposition of a quaternionic matrix
$A \in \H^{n \times p}$ with full right column rank\footnote{A quaternionic matrix $A \in \H^{n \times p}$ is said to have full right column rank if the column vectors $a_1, a_2, \dots, a_p \in \H^n$ are right linearly independent, i.e., for $x \in \H^p$, $Ax = 0 \implies x = 0$.}, normalized so that the upper triangular R-factor has positive real diagonal entries.
Then, $R^G$ is a retraction on $\St_G(p,\H^n)$.
\end{proposition}
\begin{proof}
Let $R \colon T\! \St(p,\H^n) \to \St(p,\H^n)$ be a map defined by $R_X(\eta) \coloneqq \qf(X + \eta)$
for $X \in \St(p,\H^n)$ and $\eta \in T_X \!\St(p,\H^n)$.
This is shown to be a retraction on $\St(p,\H^n)$~\cite{WangYang2026}.
For the diffeomorphism $\Phi \colon \St(p,\H^n) \to \St_G(p,\H^n)$, we have
$R^G_X(\eta) = \sqrt{G}^{-1}R_{\sqrt{G}X}(\sqrt{G}\eta)= \Phi(R_{\Phi^{-1}(X)}(\D \Phi^{-1}(X)[\eta]))$.
From general theory (Theorem~3.2 in~\cite{SatoAihara2019}), this means that $R^G$ is a retraction on the generalized quaternionic Stiefel manifold $\St_G(p,\H^n)$.
\end{proof}
\begin{remark}
For $X \in \St_G(p,\H^n)$ and $\eta \in T_X\!\St_G(p,\H^n)$, the matrix
$(\sqrt{G}(X+\eta))^H(\sqrt{G}(X+\eta)) = I_p + \eta^H G\eta$
is positive definite.
Hence, $\sqrt{G}(X+\eta)$ has full right column rank, and therefore its QR decomposition is unique.
\end{remark}
A vector transport is used, e.g., in the Riemannian conjugate gradient method~\cite{Absil2008,Sato2021,sato2022riemannian}.
For example, the map $\mathcal{T} \colon T\!\St_G(p,\H^n) \oplus T\!\St_G(p,\H^n) \to T\!\St_G(p,\H^n)$ ($\oplus$ denotes the Whitney sum) defined by
$\mathcal{T}_{\eta}(\xi) \coloneqq P_{X_+}^{M_{X_+}}(\xi)$
for $X \in \St_G(p,\H^n)$ and $\eta, \xi \in T_X \!\St_G(p,\H^n)$ with $X_+ \coloneqq R^G_X(\eta)$ is a vector transport from general theory~\cite{Absil2008}, which is computed using the orthogonal projection onto the tangent space at $X_+$.
\section{Applications and numerical results}
\label{sec:applications}
In this section, we deal with the quaternionic GEVP and quaternionic CCA as applications of optimization on the generalized quaternionic Stiefel manifold.
We demonstrate numerical experiments illustrating optimization on the generalized quaternionic Stiefel manifold.
All computations were performed in double-precision floating-point arithmetic on a computer
(Apple M4 Pro, 48 GB RAM) with MATLAB R2024b.
For quaternionic computations, we used QTFM (Version 3.7)~\cite{QTFM2025},
and for Riemannian optimization, we used Manopt 8.0~\cite{boumal2014manopt}.
\subsection{Generalized quaternionic eigenvalue problem}
\label{subsec:GEVP}
For $A, G \in \Her(n)$ with $G$ being positive definite, we say that $\lambda \in \H$ is a right generalized eigenvalue~\cite{hong2021inequalities} and $x \in \H^n - \{0\}$ is an associated right generalized eigenvector if they satisfy
\begin{equation}
\label{eq:GeneralizedEigenvalue}
Ax = Gx\lambda.
\end{equation}
In this case, any eigenvalue $\lambda$ turns out to be a real value~\cite{hong2021inequalities}, and \eqref{eq:GeneralizedEigenvalue} is, in fact, equivalent to
$Ax = \lambda Gx$.
We can formulate the problem of finding $p$ generalized eigenvectors associated with the $p$ smallest generalized eigenvalues as an optimization problem on $\St_G(p,\H^n)$ as follows.
\begin{proposition}
For $A, G \in \Her(n)$ with $G$ being positive definite, let $X^*$ be an optimal solution to the optimization problem
\begin{equation}
\label{prob:opt}
\min_{X \in \St_G(p,\H^n)} \re(\tr(X^HAXN))
\end{equation}
on $\St_G(p,\H^n)$, where $N = \diag(\mu_1, \mu_2, \dots, \mu_p)$ with $\mu_1 > \mu_2 > \cdots > \mu_p > 0$.
Then, for the generalized eigenvalue problem~\eqref{eq:GeneralizedEigenvalue},
the $r$-th column vector of $X^*$ is a generalized eigenvector associated with the $r$-th smallest generalized eigenvalue $\lambda_r$ (counted with multiplicity).
\end{proposition}
\begin{proof}
Since any generalized eigenvalue $\lambda$ in~\eqref{eq:GeneralizedEigenvalue} is a real value and commutes with any quaternionic matrix, the proof of the proposition is completed in the same way as the real case~\cite{Boumal2023,Edelman1998}.
\end{proof}
We considered Problem~\eqref{prob:opt} with $n = 1000$ and $p = 10$.
We first generated random quaternionic matrices $U, V \in \H^{n \times n}$ whose four real components were independently sampled from the standard normal distribution, and set
$A = (U + U^H)/2$ and $G = V^H V + I_n$.
We also set $N = \diag(p, p-1, \dots, 1)$.
As a Riemannian metric, we set $M_X \equiv G$ for all $X \in \St_G(p,\H^n)$.
On $\St_G(p,\H^n)$, we compared the Riemannian steepest descent (SD), conjugate gradient (CG), and trust-region (TR) methods implemented in Manopt.
The same randomly generated initial point $X_0 \in \St_G(p,\H^n)$ was used for all three solvers for a fair comparison.
The maximum number of iterations was set to $250$, and the stopping criterion was $\|\grad f(X_k)\|_{X_k} < 10^{-6}$.
For the trust-region method, we set the initial trust-region radius to $\sqrt{p}$.
All the other parameters were left at their default values in Manopt.
The Riemannian gradient norm of $f$ versus the iteration number is shown in Fig.~\ref{fig1}, which
indicates that optimization on $\St_G(p,\H^n)$ works well.
Furthermore, for $\tilde{X}$ obtained at the termination of TR and $\tilde{\Lambda} \coloneqq \tilde{X}^HA\tilde{X}$, we observed $\|\tilde{X}^H G\tilde{X}-I_p\|_F = 4.99 \times 10^{-13}$ (feasibility residual),
$\|\tilde{\Lambda} - \diag(\tilde{\Lambda})\|_F = 1.96 \times 10^{-9}$ (off-diagonal residual), and
$\|A\tilde{X} - G\tilde{X}\diag(\tilde{\Lambda})\|_F = 3.93 \times 10^{-8}$ (generalized eigenvalue residual).
All three residuals are close to zero.
\begin{figure}[thpb]
    \centering
    \includegraphics[width=.65\columnwidth]{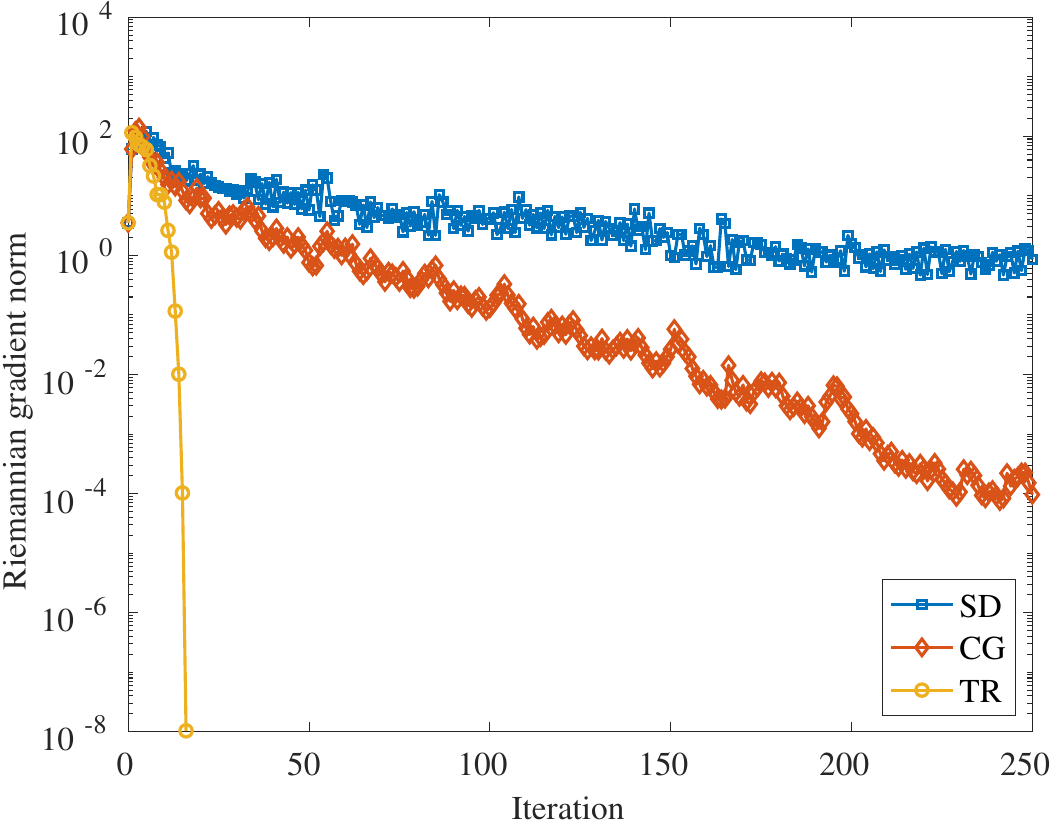}
    \caption{Riemannian gradient norm versus iterations.}
    \label{fig1}
\end{figure}
As another experiment, to illustrate that the proposed framework is not restricted to the metric $M_X \equiv G$, we also solved the same quaternionic
GEVP using a canonical-type $X$-dependent metric $M_X^{\rm can}$.
This metric is obtained from the canonical metric on the
standard quaternionic Stiefel manifold through $Y=\Phi^{-1}(X) = \sqrt{G}X$ as
\[
    M_X^{\rm can}
    =
    \sqrt{G}\left(I_n-\frac12YY^H\right)\sqrt{G} = G - \frac{1}{2}GXX^HG,
\]
which is Hermitian positive definite for any $X \in \St_G(p,\H^n)$.
Moreover, we have $(M_X^{\rm can})^{-1}=G^{-1}+XX^H$.
We compared the metric \(M_X\equiv G\) and canonical-type
metric $M_X = M_X^{{\rm can}}$ using CG with the same
initial point.
For this comparison, we report the feasibility residual $r_{\rm feas}$, off-diagonal residual $r_{\rm off}$, and generalized eigenvalue residual $r_{\rm eig}$.
\begin{table}[t]
\caption{GEVP solved by CG with different Riemannian metrics.}
\label{tab:gevp-metrics}
\centering
\begin{tabular}{lcccc}
\hline
$M_X$ & time [s]
& \(r_{\rm feas}\) & \(r_{\rm off}\) & $r_{\rm eig}$\\
\hline
$G$ & 36.57 & \(4.77\times10^{-13}\) & \(1.74\times10^{-5}\) & $9.39\times10^{-4}$\\
$M_X^{{\rm can}}$ & 42.69 & \(5.20\times10^{-13}\) & \(6.06\times10^{-6}\) & $2.79 \times 10^{-4}$\\
\hline
\end{tabular}
\end{table}
CG methods with both metrics were terminated after 250 iterations and reached essentially the same objective value as $-280.5710$.
Table~\ref{tab:gevp-metrics} shows that both sufficiently preserve the generalized quaternionic orthogonality constraint.
In this instance,
the canonical-type metric gives smaller off-diagonal and generalized
eigenvalue residuals, although it requires more computational time
because applying \((M_X^{\rm can})^{-1}\) involves solving a linear
system with \(G\). The purpose of this comparison is not to claim that
the canonical-type metric is uniformly superior, but to demonstrate
that the proposed \(M_X\)-dependent framework can be used with a
genuinely \(X\)-dependent metric.
\subsection{CCA with quaternion-linear projections}
We next discuss a version of the quaternionic CCA problem that naturally leads to the product of generalized quaternionic Stiefel
manifolds.
In the real case, CCA with multiple canonical directions can be formulated as an optimization problem on the product of generalized Stiefel manifolds~\cite{ablin2024infeasible,Yger2012}.
Quaternion extensions of CCA and other statistical methods have also been studied in the widely-linear framework for quaternion random vectors~\cite{via2010properness}.
For jointly Q-proper vectors, the optimal CCA projections reduce to conventional linear processing (see~\cite{via2010properness} for details of the terminology).
Motivated by these results, we consider the CCA problem restricted to quaternion-linear projections below.
In what follows, let \(x\in\mathbb H^{m}\) and \(y\in\mathbb H^{n}\) be zero-mean quaternion-valued random vectors.
Here, we note that quaternions can be used to express data with a $3$- or $4$-dimensional structure, e.g., data in a real $3$-dimensional space (with scalars for $4$-dimensional version).
For example, let $q=T+v_xi+v_yj+v_zk \in\mathbb H$, where $T$ is temperature (a scalar) and $(v_x,v_y,v_z)$ is the three-dimensional wind velocity.
Such quaternion-valued
representations have been used for joint modeling and forecasting of three-dimensional wind and atmospheric parameters~\cite{took2011quaternion}.
Suppose that $x\in\mathbb H^{m}$ consists of such quaternion-valued variables at several spatial locations, and $y\in\mathbb H^{n}$ consists of the corresponding outputs of
a numerical prediction model.
We consider quaternionic linear combinations of the form $u=U^Hx$ and $v=V^Hy$, where $U\in\mathbb H^{m\times p}$ and $V\in\mathbb H^{n\times p}$.
Let $C_{xx}=\mathbb E[xx^H]$, $C_{yy}=\mathbb E[yy^H]$, and $C_{xy}=\mathbb E[xy^H]$, where we assume that $C_{xx}$ and $C_{yy}$ are positive definite.
Then, we have $\mathbb E[uu^H]=U^HC_{xx}U$, $\mathbb E[vv^H]=V^HC_{yy}V$, and $\mathbb E[uv^H]=U^HC_{xy}V$.
The normalization conditions for $u$ and $v$ are $U^HC_{xx}U=I_p$ and $V^HC_{yy}V=I_p$, which are equivalent to $(U, V) \in \St_{C_{xx}}(p,\H^m) \times \St_{C_{yy}}(p,\H^n) \eqqcolon \mathcal{M}$.
The generalized
quaternionic Stiefel manifolds thus naturally arise.
Following the formulation of real CCA~\cite{Yger2012}, we consider the following optimization problem on $\mathcal{M}$:
\begin{equation}
\max_{(U, V) \in \mathcal{M}}\re(\tr(U^HC_{xy}VN)),
\label{qCCA}
\end{equation}
with $N=\diag(\nu_1, \nu_2, \ldots,\nu_p)$ ($\nu_1> \nu_2 > \cdots > \nu_p>0$).
\begin{remark}
It is worthwhile to clarify the relation between \eqref{qCCA} and applying real CCA to the real representation of quaternion-valued data.
If quaternion-valued data are identified with real vectors and
ordinary real CCA is applied, the four components of each
quaternion and the resulting
canonical variables are real-valued.
In contrast, \eqref{qCCA} uses
quaternion-linear projections $u=U^Hx$ and $v=V^Hy$ to
directly extract quaternion-valued canonical variables.
The real
representation of a quaternion-linear projection by $U$ is indeed a real
linear map, but it is not arbitrary: it has a special
block structure determined by the four real component matrices of $U$.
Thus, \eqref{qCCA} is a structured formulation
that imposes quaternion-linear structure on the admissible projections. \end{remark}
For numerical experiments on the CCA problem~\eqref{qCCA}, we used a synthetic instance with a known
optimum.
We generated random Hermitian positive definite matrices
$C_{xx}=H_x^HH_x+I_{m}$ and $C_{yy}=H_y^HH_y+I_{n}$.
Then, we generated $(U_\star, V_\star)\in\mathcal{M}$ and set $C_{xy}=C_{xx}U_\star\Sigma V_\star^HC_{yy}$ with $\Sigma=\diag(0.95,0.80,0.65,0.50,0.35)$, $m=80$, $n=70$, $p=5$, and $N=\diag(5,4,3,2,1)$.
Since $U_\star^HC_{xx}U_\star=I_p$ and $V_\star^HC_{yy}V_\star=I_p$, the optimal value of
$\max_{(U,V)\in \mathcal{M}}\re(\tr(U^HC_{xy}VN))$ is $\tr(\Sigma N)=11.25$.
We solved the equivalent minimization problem, i.e., the minimization of $-\re(\tr(U^HC_{xy}VN))$, using SD, CG, and TR with the same initial point.
The maximum number of iterations was 250, and the
stopping tolerance for the gradient norm was $10^{-6}$.
The Riemannian metrics were $M_U\equiv C_{xx}$ and $M_V\equiv C_{yy}$.
Table~\ref{tab:qcca} reports the relative objective gap
$r_{\rm gap} = |\re(\tr(U^HC_{xy}VN)) -\tr(\Sigma N)| / \tr(\Sigma N)
    $
and feasibility residual $r_{\rm feas} = \max\{\|U^HC_{xx}U-I_p\|_F,\|V^HC_{yy}V-I_p\|_F\}$.
\begin{table}[t]
\caption{CCA with quaternion-linear projections.}
\label{tab:qcca}
\centering
\begin{tabular}{lcccc}
\hline
method & iter. & time [s] & $r_{\rm gap}$ & $r_{\rm feas}$ \\
\hline
SD & 250 & 3.76 &$1.76\times10^{-7}$ & $7.45\times10^{-16}$ \\
CG & 100 & 2.53 & $1.31\times10^{-13}$ & $7.02\times10^{-16}$ \\
TR & 10 & 0.94 & $3.16\times10^{-16}$ & $6.46\times10^{-16}$ \\
\hline
\end{tabular}
\end{table}
SD reached the maximum number of iterations, but already
attained a small objective gap.
CG and TR satisfied the gradient tolerance.
These results show that the proposed framework can solve the
CCA problem on $\mathcal{M}$
while preserving the covariance-weighted orthogonality constraints.
\section{Concluding remarks}
\label{sec:conclusion}
In this paper, we defined the generalized quaternionic Stiefel manifold and examined its geometry.
The results are analogous to those for the real and complex generalized Stiefel manifolds when the base field is replaced with $\mathbb{H}$, and also analogous to those for the quaternionic Stiefel manifold when the standard inner product on $\mathbb{H}^n$ is replaced with one defined by a general Hermitian positive definite matrix $G$.
Specifically, we regarded $\H^{n \times p}$ as a real $4np$-dimensional vector space, endowed it with a Riemannian metric defined by a general positive definite matrix $M_X$ at each $X$, and regarded $\St_G(p,\H^n)$ as a real Riemannian submanifold of $\H^{n \times p}$.
Then, we derived explicit formulas for various concepts required for Riemannian optimization on the manifold.
As applications, the quaternionic GEVP and CCA problem with quaternion-linear projections were formulated as optimization problems on the generalized quaternionic Stiefel manifold and on the product of two generalized quaternionic Stiefel manifolds, respectively.
Numerical results indicated that optimization algorithms on the manifolds work well and appropriately solve the optimization problems.
\section*{Acknowledgment}
The author would like to thank the anonymous reviewers for their valuable comments that helped improve the paper.

\end{document}